\documentclass[11pt,reqno]{amsart}
\usepackage[latin1]{inputenc}
\usepackage{color}
\usepackage{amsmath}
\usepackage{amssymb}
\usepackage{pdfpages}
\usepackage{listings}
\usepackage{float}
\floatstyle{ruled} 
\usepackage{caption}
\restylefloat{figure}
\newtheorem{thm}{Theorem}[section]
\newtheorem{cor}[thm]{Corollary}
\newtheorem{lem}[thm]{Lemma}
\newtheorem{prop}[thm]{Proposition}

\theoremstyle{definition}
\newtheorem{defin}[thm]{Definition}
\newtheorem{rem}[thm]{Remark}

\numberwithin{equation}{section}

\def \lr {{\quad\Leftrightarrow\quad}}

\def\R{{\mathbb R}}
\def\N{{\mathbb N}}
\def\SS{{\mathcal S}}

\makeatletter
\@ifpackageloaded{showkeys}{\oddsidemargin=3cm\evensidemargin=3cm}{\oddsidemargin=0.7cm\evensidemargin=0.7cm}\makeatother
\newcommand{\supp}{\operatorname{supp}}

\newcommand{\dif}{\,\mathrm{d}}

\def\be  {\begin{equation}} 
\def\ee  {\end{equation}}
\def\ba  {\begin{eqnarray}} 
\def\ea  {\end{eqnarray}}
\def\baa {\begin{eqnarray*}} 
\def\eaa {\end{eqnarray*}}

\newcommand{\mathematica}{\scshape Mathe\-matica\normalfont}

\allowdisplaybreaks

\begin{document}

\title[Pointwise estimates for B-spline Gram matrix inverses]{Pointwise estimates for B-spline Gram matrix inverses}
\author[M. Passenbrunner]{Markus Passenbrunner}
\address{Institute of Analysis, Johannes Kepler University Linz, Austria, 4040 Linz, Alten\-berger Straße 69}
\email{markus.passenbrunner@jku.at}
\date{\today}
\begin{abstract}
We present a new method for proving a certain geometric-decay inequality for entries of inverses of B-spline Gram matrices, which is given in \cite{PassenbrunnerShadrin2013}. 
\end{abstract}
\maketitle

\section{Introduction}\label{sec:introduction}
Let $k,m \in \N$ be fixed. We define a knot-sequence 
$\Delta=(t_i)_{i=1}^{m+k}$ such that 
\baa
   t_i < t_{i+1}, \qquad   0 = t_1 = \cdots = t_k, \qquad t_{m+1} = \cdots = t_{m+k} = 1.
\eaa
For convenience, we set $t_i=0$ for $i\leq 0$ and $t_i=1$ for $i\geq m+k+1$.
Let $(N_{i,k})_{i=1}^m=(N_i)_{i=1}^m$ be the sequence of $L_\infty$-normalized B-splines 
of order $k$ on $\Delta$. 
We use the notations
$$
   |\Delta| := \max_i (t_{i+1}-t_i), \qquad
  J_{ij} := [t_{\min(i,j)},t_{\max(i,j) + k}], \qquad \eta_{ij}=|J_{ij}|.\qquad
$$
Define the B-spline Gram matrix $A=(\langle N_i,N_j\rangle)_{i,j=1}^m$ and its inverse $B:=(b_{i,j})_{i,j=1}^m:=A^{-1}$. In this note, we will present a new method of proof of an inequality of the following type:
\begin{thm}\label{thm:main}
The entries of $B$ satisfy the estimate 
\[
    |b_{i,j}| \le K \gamma^{|i-j|} \eta_{ij}^{-1},\qquad {1}\leq {i},{j}\leq {n},
\]
where $K>0$ and $\gamma \in (0,1)$ are constants that depend only on $k$ and not on the partition $\Delta$.
\end{thm}
This result was proved in \cite{PassenbrunnerShadrin2013} for general spline orders $k$ and has many consequences, for instance 
\begin{enumerate}
\item a.e. convergence of the orthogonal projection $P_\Delta f$ as $|\Delta|\to 0$ arbitrarily, where $P_\Delta$ is the orthogonal projection operator onto $\operatorname{span}\{N_i:1\leq i\leq m\}$ \cite{PassenbrunnerShadrin2013},
\item unconditionality of orthonormal spline series in reflexive $L^p$ spaces \cite{Passenbrunner2013}.
\end{enumerate}
In order to compare the methods used in \cite{PassenbrunnerShadrin2013} and those used here, we note that the proof in \cite{PassenbrunnerShadrin2013}  strongly uses Shadrin's theorem \cite{Shadrin2001} that $\|P_\Delta\|_\infty\leq c$ for some constant $c$ that depends only on $k$ and not on $\Delta$.
Our proof does not use Shadrin's theorem, but we are able to prove Theorem \ref{thm:main} only for the spline orders $k=2$ and $k=3$.
We note that for $k=2$, i.e. in the piecewise linear case, there is a very direct argument by Z. Ciesielski to obtain Theorem \ref{thm:main}, which is presented in \cite{KashinSaakyan1989} (cf. the results in \cite{Ciesielski1963,Ciesielski1966}). However, this proof is strictly limited to $k=2$. 
One additional advantage of our approach is, that for $k=3$, we get sharper constants $K$ and $\gamma$ than the ones obtained in \cite{PassenbrunnerShadrin2013}.

This article is structured as follows. In Section \ref{sec:prel}, we collect a few preliminaries about matrices and B-splines. In Section \ref{sec:it} we derive a simple iteration formula for inverse matrices, which is the basis of our method of proving Theorem \ref{thm:main}. Next, we use this iteration formula in Section \ref{sec:linear} to give a new proof of Theorem \ref{thm:main} for $k=2$. Finally, in Section \ref{sec:quad}, we show how our iterative method can be used to prove Theorem \ref{thm:main} for $k=3$.
\section{Preliminaries}\label{sec:prel}
\subsection{The Sherman-Morrison formula}
We have the following formula for the inverse of a rank $j$ perturbation of a given matrix $A$.
\begin{thm}\label{thm:shermanmorrison}\cite{ShermanMorrison1950, Woodbury1950} 
Let $A$ be an invertible $m\times m$ matrix and $U,V$ $m\times j$ matrices. If the $j\times j$ matrix 
\[
1+V^T A^{-1} U
\]
is invertible, then we have
\[
(A+UV^T)^{-1}=A^{-1}-A^{-1}U(1+V^T A^{-1}U)^{-1} V^T A^{-1}.
\]
\end{thm}
In applications of this formula, $j$ is typically much smaller than $m$. We will apply it for the choice $j=2$.

\subsection{B-splines}\label{sec:bsplines}
The B-spline functions $N_i$ have the properties
$$
   \supp N_i = [t_i, t_{i+k}]\,, \qquad N_i \ge 0\,, \qquad
   \sum_i N_i \equiv 1\,.
$$
Moreover, we have the recursion formula for B-splines:
\begin{equation*}
N_{i,k+1}(x)=\frac{x-t_i}{t_{i+k}-t_i} N_{i,k}+\frac{t_{i+k+1}-x}{t_{i+k+1}-t_{i+1}}N_{i+1,k},
\end{equation*}
and the $L^1$-norm of $N_{i,k}$ is given by 
\begin{equation*}
\left\|N_{i,k}\right\|_1=\int_0^1 N_{i,k}(x)\dif x=\frac{t_{i+k}-t_i}{k}.
\end{equation*}

For each $\Delta$, we define then the space $\SS_k(\Delta_n)$ of splines 
of order $k$ with knots $\Delta$ as a linear span of $(N_i)$, namely 
$$
   s \in \SS_k(\Delta) \lr s = \sum_{i=1}^n c_i N_i\,,\quad c_i \in \R\,,
$$
so that $\SS_k(\Delta)$ is the space of piecewise polynomial functions
of degree $k-1$, with $k-2$ continuous derivatives at $t_i$.

Observe that $\eta_{ij}=t_{\max(i,j)+k}-t_{\min(i,j)}$ satisfies the inequality
\begin{equation}\label{eq:etaineq}
\eta_{j,n+1}\leq \eta_{jn}\frac{t_{n+k+1}-t_{n+1}}{t_{n+k}-t_{n+1}},\qquad j\leq n.
\end{equation}
This follows from the elementary inequality
\[
a+b+c\leq \frac{(a+b)(a+c)}{a}
\]
for positive real numbers $a,b,c$ with the choices $a=t_{n+k}-t_{n+1}$, $b=t_{n+1}-t_j$ and $c=t_{n+k+1}-t_{n+k}$.

In order to keep future expressions involving distances of points in $\Delta$ simple, we introduce the following notation:
\begin{equation}
(\ell,n)_j:=(\ell n)_j:=t_{j+\ell}-t_{j+n}
\label{eq:defnot}
\end{equation}
for integer parameters $\ell,n,j$.

\subsection{Total positivity}
We denote by $Q_{\ell,n}$ the set of strictly increasing sequences of $\ell$ integers from the set $\{1,\dots,n\}$.
Let $A$ be an $n\times n$-matrix. For $\alpha,\beta\in Q_{\ell,n}$, we denote by $A[\alpha;\beta]$ the submatrix of $A$ consisting of the rows indexed by $\alpha$ and the columns indexed by $\beta$. Furthermore we let $\alpha'$ (the complement of $\alpha$) be the uniquely determined element of $Q_{n-\ell,n}$ that consists of all integers in $\{1,\dots,n\}$ not occurring in $\alpha$. In addition, we use the notation $A(\alpha;\beta):=A[\alpha';\beta']$. 
\begin{defin}
Let $A$ be an $n\times n$-matrix. $A$ is called \emph{totally positive}, if 
\begin{equation}
\det A[\alpha;\beta]\geq 0,\quad \text{for }\alpha,\beta\in Q_{\ell,n},\,1\leq \ell\leq n.
\label{eq:tpdef}
\end{equation}
\end{defin}
The cofactor formula $b_{i,j}=(-1)^{i+j}\det A(j;i)/\det A$ for the inverse $B=(b_{i,j})_{i,j=1}^n$ of the matrix $A$ leads to
\begin{prop}\label{prop:checkerboard}
Inverses $B=(b_{i,j})$ of totally positive matrices $A=(a_{i,j})$ have the checkerboard property. This means that
\begin{equation*}
(-1)^{i+j} b_{i,j}\geq 0\quad \text{for all }i,j.
\end{equation*}
\end{prop}
The theory of totally positive matrices can be applied to our Gram matrices of B-splines, since we have
\begin{thm}[\cite{deBoor1968}]\label{thm:gramtp}
The Gram matrix of B-splines $A=(\langle N_{i,k},N_{j,k}\rangle)_{i,j=1}^{m}$ of arbitrary order $k$ and arbitrary partition $\Delta$ is totally positive. 
\end{thm}
This theorem is a consequence of the so called basic composition formula (Equation (2.5), Chapter 1 in \cite{Karlin1968}) and the fact that the kernel $N_{i,k}(x)$, depending on the variables $i$ and $x$, is totally positive (Theorem 4.1, Chapter 10 in \cite{Karlin1968}). Thus the inverse of $A$ possesses the checkerboard property by the above proposition. 

\section{An iteration formula for inverses of matrices}\label{sec:it}
In this section, we use the Sherman-Morrison formula to obtain an iterative expression for inverse matrices.
Let $A=(a_{i,j})_{i,j=1}^{m}$ be an invertible $m\times m$-matrix and define 
\[
A_n=(a_{i,j})_{i,j=1}^{n}\quad \text{for }1\leq n\leq m
\]
to be the $n\times n$ matrix consisting of the first $n$ rows and the first $n$ columns of $A$. Furthermore set $B_n:=A_n^{-1}$ if $A_n$ is invertible and let $(b^n_{i,j})_{i,j=1}^n:=B_n$. Then $A_{n+1}$ may be written as the sum of two matrices as follows:
\begin{equation}\label{eq:decompAn}
A_{n+1}=\begin{pmatrix}
A_n & \mathbf{0}_{n\times 1} \\
\mathbf{0}_{1\times n} & a_{n+1,n+1}
\end{pmatrix}
+\begin{pmatrix}
\mathbf{0}_{n\times n} & u^n \\
(v^n)^T & 0
\end{pmatrix}=:S_n+T_n,
\end{equation}
where $v^n\in\mathbb{R}^n$ is the column vector $(a_{n+1,1},\dots,a_{n+1,n})^T$ and $u^n$ is the column vector $(a_{1,n+1},\dots,a_{n,n+1})^T$. $T_n$ is a rank $2$ matrix that can be written as the product $T_n=U_nV_n^T$ where $U_n$ and $V_n$ are $(n+1)\times 2$ matrices and defined as 
\[
U_n=\begin{pmatrix}
u^n & \mathbf{0}_{n\times 1} \\
0 & 1
\end{pmatrix},\quad V_n=\begin{pmatrix}
\mathbf{0}_{n\times 1} & v^n \\
1 & 0
\end{pmatrix}.
\]

If we apply Theorem \ref{thm:shermanmorrison} to the above decomposition \eqref{eq:decompAn} of $A_{n+1}$, we get by some easy computations:
\begin{cor}\label{prop:inv}
Let $1\leq n\leq m$. Additionally, suppose that $A_n,B_n$ and $v^n,u^n$ are defined as above and set $v=v^n, u=u^n$. If $A_n$ is invertible, $a_{n+1,n+1}\neq 0$ and $a_{n+1,n+1}-v^T B_n u\neq 0$ we have the following formula for $B_{n+1}=A_{n+1}^{-1}$:
\[
B_{n+1}=\begin{pmatrix} B_n & \mathbf{0}_{n\times 1} \\
\mathbf{0}_{1\times n} & 0\end{pmatrix}
+\frac{1}{a_{n+1,n+1}-v^T B_n u}\begin{pmatrix}
B_nuv^TB_n & -B_nu \\
-v^T B_n & 1
\end{pmatrix},
\]  
\end{cor}

We employ this corollary in the cases where the matrix $A$ is a symmetric tridiagonal or a symmetric $5$-banded matrix. 
\subsection{Tridiagonal matrices}\label{sec:tridiag}
Let $A$ be a symmetric tridiagonal $m\times m$-matrix.
Using Corollary \ref{prop:inv} on $A$, we get for $1\leq n\leq m-1$
\begin{equation}
B_{n+1}=\begin{pmatrix}
B_n & \mathbf{0}_{n\times 1} \\
\mathbf{0}_{1\times n} & 0
\end{pmatrix}
+
(a_{n+1,n+1}-b_{n,n}^n a_{n,n+1}^2)^{-1}C_{n+1},
\end{equation}
where $C_{n+1}$ is given by
\[
(C_{n+1})_{i,j}=\begin{cases}
b_{i,n}^n b_{j,n}^n a_{n,n+1}^2 & \text{if }1\leq i,j\leq n, \\
-b_{i,n}^n a_{n,n+1} & \text{if } 1\leq i\leq n, j=n+1,\\
-b_{j,n}^n a_{n,n+1} & \text{if } 1\leq j\leq n, i=n+1, \\
1 &\text{if }i=j=n+1.
\end{cases}
\]

\subsection{5-banded matrices}\label{sec2band}
Let $A=(a_{i,j})$ be a symmetric $5$-banded matrix. This means that $a_{i,j}=0$ for $|i-j|>2$.
Using Corollary \ref{prop:inv} on $A$, we get for $2\leq n\leq m-1$
\begin{equation}\label{eq:itquadr}
B_{n+1}=\begin{pmatrix}
B_n & \mathbf{0}_{n\times 1} \\
\mathbf{0}_{1\times n} & 0
\end{pmatrix}
+
b_{n+1,n+1}^{n+1}C_{n+1},
\end{equation}
where 
\begin{equation}
\begin{aligned}
b_{n+1,n+1}^{n+1}=&\big(a_{n+1,n+1}-b_{n,n}^n a_{n,n+1}^2-2b_{n-1,n}^na_{n-1,n+1}a_{n,n+1}\\
&-b_{n-1,n-1}^n a_{n-1,n+1}^2\big)^{-1}
\end{aligned}
\label{eq:bnnnquadr}
\end{equation}
and $C_{n+1}$ is given by
\begin{equation}
\label{eq:Cnquadr}
(C_{n+1})_{i,j}=\begin{cases}
-b_{i,n}^n a_{n,n+1}-b_{i,n-1}^n a_{n-1,n+1} & \text{if } 1\leq i\leq n, j=n+1, \\
-b_{j,n}^n a_{n,n+1}-b_{j,n-1}^n a_{n-1,n+1} & \text{if } 1\leq j\leq n, i=n+1, \\
(C_{n+1})_{i,n+1}(C_{n+1})_{j,n+1} & \text{if }1\leq i,j\leq n,\\
1 &\text{if }i=j=n+1.
\end{cases}
\end{equation}
The following two lemmas 
are independent of the special form of matrices considered in the next sections. We will use them in Section \ref{sec:quad} to estimate inverses of Gram matrices corresponding to splines of order 3.
The crucial fact about the following lemma is that inequality \eqref{eq:quadrgeneral1} depends only on the matrix $A$ and on entries $b_{n,n}^n, b_{n-1,n-1}^{n-1}$ of the matrices $B_n,B_{n-1}$ respectively.

\begin{lem}\label{lem:bnnnquad}
Let $A$ be a symmetric $5$-banded $m\times m$-matrix such that $B_n=A_n^{-1}$ is checkerboard for $1\leq n\leq m$. Then, the inequality
\begin{equation}
\begin{aligned}
b_{n+1,n+1}^{n+1}\leq\;& \Big(a_{n+1,n+1}-b_{n,n}^n a_{n,n+1}\big(a_{n,n+1}-\frac{2a_{n,n}a_{n-1,n+1}}{a_{n-1,n}}\big)-2\frac{a_{n,n+1}a_{n-1,n+1}}{a_{n-1,n}}\\
&-a_{n-1,n+1}^2b_{n-1,n-1}^{n-1}(1+b_{n,n}^n b_{n-1,n-1}^{n-1} a_{n-1,n}^2)\Big)^{-1}
\end{aligned}
\label{eq:quadrgeneral1}
\end{equation}
holds for $2\leq n\leq m-1$.
\end{lem}
\begin{proof}
By \eqref{eq:bnnnquadr} and the checkerboard property of $B_n$, $b_{n+1,n+1}^{n+1}$ is given by 
\begin{equation}
\label{eq:bnnn:0}
b_{n+1,n+1}^{n+1}=(a_{n+1,n+1}-b_{n,n}^n a_{n,n+1}^2+2 |b_{n-1,n}^n|a_{n,n+1}a_{n-1,n+1}-b_{n-1,n-1}^n a_{n-1,n+1}^2)^{-1}.
\end{equation}
Another consequence of the identities \eqref{eq:itquadr}\,--\,\eqref{eq:Cnquadr} is
\begin{align*}
b_{i,j}^n&=b_{i,j}^{n-1}+b_{n,n}^n (b_{i,n-1}^{n-1}a_{n-1,n}+b^{n-1}_{i,n-2}a_{n-2,n}) (b_{j,n-1}^{n-1}a_{n-1,n}+b^{n-1}_{j,n-2}a_{n-2,n})
\end{align*}
for $1\leq i,j\leq n-1$. The choice $i=j=n-1$ in these equations together with the checkerboard property of $B_n$ yield the estimate
\begin{equation}\label{eq:bnnn:1}
b_{n-1,n-1}^n\leq b_{n-1,n-1}^{n-1}(1+b_{n,n}^n b_{n-1,n-1}^{n-1} a_{n-1,n}^2).
\end{equation}
The defining property of the inverse matrix $B_n=A_n^{-1}$, the fact that $A_n$ is $2$-banded and the checkerboard property of $B_n$ again imply
\begin{equation}
\begin{aligned}
|b_{n-1,n}^n|&=a_{n-1,n}^{-1}(b_{n,n}^n a_{n,n}+b_{n-2,n}^n a_{n-2,n}-1) \\
&\geq a_{n-1,n}^{-1}(b_{n,n}^n a_{n,n}-1).
\end{aligned}
\label{eq:bnnn:2}
\end{equation}
Finally, inserting \eqref{eq:bnnn:1} and \eqref{eq:bnnn:2} in \eqref{eq:bnnn:0} yields the conclusion of the lemma.
\end{proof}
\begin{lem}\label{lem:estbjnn}
Let $A$ be as in Lemma \ref{lem:bnnnquad} and $2\leq n\leq m$. Then, if $1\leq j\leq n-1$, the estimate
\begin{equation}
|b_{jn}^n|\leq |b_{j,n-1}^{n-1}|b_{n,n}^n a_{n-1,n}
\label{eq:bjnnbad}
\end{equation}
holds. Additionally for $3\leq n\leq m$ and $1\leq j\leq n-2$ we have 
\begin{equation}
|b_{j,n}^n|\leq |b_{j,n-1}^{n-1}|b_{n,n}^n\Big(a_{n-1,n}-\frac{a_{n-2,n}a_{n-1,n-1}}{a_{n-2,n-1}}\Big).
\label{eq:bjnnbetter}
\end{equation}
\end{lem}

\begin{proof}
First, we note that \eqref{eq:itquadr}\,--\,\eqref{eq:Cnquadr} yield
\begin{equation}
b_{j,n}^n=-b_{n,n}^n(b_{j,n-2}^{n-1} a_{n-2,n}+b_{j,n-1}^{n-1} a_{n-1,n})\quad \text{if }1\leq j\leq n-1.
\label{eq:bnjn:0}
\end{equation}
Since $B_n$ is checkerboard, \eqref{eq:bjnnbad} follows for $1\leq j\leq n-1$.
By the defining property of $B_n=A_n^{-1}$ we obtain
\begin{equation}
b_{j,n-3}^{n-1} a_{n-3,n-1}+ b_{j,n-2}^{n-1} a_{n-2,n-1}+b_{j,n-1}^{n-1} a_{n-1,n-1}=\delta_{j,n-1}.
\label{eq:bnjn:1}
\end{equation}
Thus we get, under the assumption $1\leq j\leq n-2$,
\begin{equation}
\begin{aligned}
|b_{j,n-2}^{n-1}|&=a_{n-2,n-1}^{-1}(a_{n-3,n-1}|b_{j,n-3}^{n-1}|+a_{n-1,n-1} |b_{j,n-1}^{n-1}|) \\
&\geq a_{n-2,n-1}^{-1}a_{n-1,n-1} |b_{j,n-1}^{n-1}|, 
\end{aligned}
\label{eq:bnjn:2}
\end{equation}
since $B_{n-1}$ is checkerboard. Now use the checkerboard property of $B_n$ in \eqref{eq:bnjn:0} and insert estimate \eqref{eq:bnjn:2} in \eqref{eq:bnjn:0} to conclude the assertion of the lemma.
\end{proof}

\section{Piecewise linear splines}\label{sec:linear}
We now apply Corollary \ref{prop:inv} to the case of piecewise linear continuous splines to get geometric estimates for the entries of their Gram matrix inverse. The purpose of the following presentation is twofold: first, the simpler case $k=2$ illustrates the basic proof steps that are essentially the same as for $k=3$. Secondly, we give a new proof of the result in \cite{KashinSaakyan1989}.

In this section, $N_i$ is the unique piecewise linear continuous function on the unit interval $[0,1]$ such that
\[
N_i(t_j)=\delta_{i,j}\quad \text{for }1\leq j\leq m.
\]

Now we consider the Gram matrix $A=(a_{i,j})_{i,j=1}^m =(\langle N_i,N_j\rangle)_{i,j=1}^m$ obtained from these functions.
Using the special form of the piecewise linear B-splines $N_i$, we get
\begin{align}
a_{i,i}&=(20)_i/3&\text{if } 1\leq i\leq m, \label{eq:gram1} \\
a_{i,i+1}=a_{i+1,i}&=(21)_i/6 &\text{if } 1\leq i\leq m-1, \\
a_{i,j}&=0& \text{if }|i-j|>1 \label{eq:gram2}.
\end{align}
Note that $A$ is symmetric and tridiagonal, therefore we may apply the formulas of Section \ref{sec:tridiag} to $A$. If we do and insert the expressions \eqref{eq:gram1}\,--\,\eqref{eq:gram2} for the Gram matrix $A$, we obtain
\begin{equation}\label{eq:invlinear}
B_{n+1}=\begin{pmatrix}
B_n & \mathbf{0}_{n\times 1} \\
\mathbf{0}_{1\times n} & 0
\end{pmatrix}
+
3\Big((20)_{n+1}-\frac{1}{12}b_{n,n}^n (21)_{n}^2\Big)^{-1}C_{n+1},
\end{equation}
where
\begin{equation}
\label{eq:invlinear1}
(C_{n+1})_{i,j}=\begin{cases}
b_{i,n}^n b_{j,n}^n (21)_n^2/36 & \text{if }1\leq i,j\leq n, \\
-b_{i,n}^n (21)_n/6 & \text{if } 1\leq i\leq n, j=n+1, \\
-b_{j,n}^n (21)_n/6 & \text{if } 1\leq j\leq n, i=n+1, \\
1 &\text{if }i=j=n+1.
\end{cases}
\end{equation}

\begin{lem}\label{lem:lowright}
Let the matrix $A$ be as above and  $1\leq n \leq m$. Then we have the estimates
\begin{equation}
\label{eq:bnnnlinear}
\frac{3}{(20)_n}\leq b_{n,n}^n\leq \frac{3}{\frac{3}{4}(10)_n+(21)_n}\leq \frac{4}{(20)_n}.
\end{equation}
\end{lem}
\begin{proof}We proceed by induction on $n$. For $n=1$, we have by definition of $B_1$ and~\eqref{eq:gram1}
\begin{equation*}
\frac{3}{(20)_1}=b_{1,1}^1=3\Big(\frac{3}{4}(10)_1+(21)_1\Big)^{-1},
\end{equation*}
since $(10)_1=0$ and thus equality on both sides of \eqref{eq:bnnnlinear}.
We know from Theorem \ref{thm:gramtp} that $A$ is totally positive, so a fortiori $A_n$ is totally positive for every $1\leq n\leq m$. Therefore, Proposition \ref{prop:checkerboard} yields $b_{n+1,n+1}^{n+1}\geq 0$. So we see the lower estimate in \eqref{eq:bnnnlinear} immediately by glancing at formula \eqref{eq:invlinear} for $b_{n+1,n+1}^{n+1}$. For the upper estimate, we obtain as a consequence of the inductive hypothesis
\begin{align*}
b_{n+1,n+1}^{n+1}&=3\Big((20)_{n+1}-\frac{1}{12}b_{n,n}^n (21)_{n}^2\Big)^{-1} \\
 &\leq 3\Big((20)_{n+1}-\frac{3}{12}\big(\frac{3}{4}(10)_n+(21)_n\big)^{-1} (21)_{n}^2\Big)^{-1} \\
 &\leq 3\Big((20)_{n+1}-\frac{1}{4}(21)_n\Big)^{-1} \\
 &=3\Big(\frac{3}{4}(10)_{n+1}+(21)_{n+1}\Big)^{-1},
\end{align*}
thus the conclusion of the lemma.
\end{proof}

\begin{lem}\label{lem:lastline}
Let the matrix $A$ be as above and  $1\leq n \leq m$. Then, the elements $b_{j,n}^n$ of $B_n$ satisfy the geometric estimate
\begin{equation}
\label{eq:geomlinear}
|b_{j,n}^n|\leq \frac{4q^{n-j}}{\eta_{jn}}\quad \text{for all }1\leq j\leq n,
\end{equation}
where $q=2/3$.
\end{lem}
\begin{proof}
We infer from \eqref{eq:invlinear} and \eqref{eq:invlinear1} that
\[
b_{j,n+1}^{n+1}=-\frac{1}{6} b_{n+1,n+1}^{n+1} b_{j,n}^n (10)_{n+1}.
\]
Invoking Lemma \ref{lem:lowright} for $b_{n+1,n+1}^{n+1}$ we get further
\begin{equation}\label{eq:geomlinear1}
|b_{j,n+1}^{n+1}|\leq \frac{2}{3}\frac{(10)_{n+1}}{(20)_{n+1}}|b_{j,n}^n|.
\end{equation}
Now we note that by Lemma \ref{lem:lowright}, inequality \eqref{eq:geomlinear} is true for $n=j$. For general $n>j$, inequality \eqref{eq:geomlinear} follows from \eqref{eq:geomlinear1} by induction using inequality \eqref{eq:etaineq}.
\end{proof}

\begin{thm}\label{thm:lingeom}
Let the matrix $A$ be as above and  $1\leq n \leq m$. Then, the elements of $B_n=A_n^{-1}$ satisfy the geometric estimate
\[
|b_{i,j}^n|\leq \frac{36}{5}\frac{q^{|i-j|}}{\eta_{ij}}\quad\text{for all } 1\leq i,j\leq n,
\]
where $q=2/3$.
\end{thm}
\begin{proof}
We first observe that it is sufficient to prove the theorem for the parameter choices $i\leq j\leq n-1$, since $B_n$ is symmetric and the case $j=n$ is already covered by Lemma \ref{lem:lastline}.
Equations \eqref{eq:invlinear} and \eqref{eq:invlinear1} yield
\[
b_{i,j}^{n}=b_{i,j}^{n-1}+ b_{n,n}^n b_{i,n-1}^{n-1} b_{j,n-1}^{n-1}\frac{(10)_n^2}{36},
\]
so by induction
\[
b_{i,j}^n=b_{i,j}^j + \sum_{\ell=j}^{n-1}b_{\ell+1,\ell+1}^{\ell+1}b_{i,\ell}^\ell b_{j,\ell}^\ell \frac{(10)_{\ell+1}^2}{36}.
\]
Using now Lemmas \ref{lem:lowright} and \ref{lem:lastline} we get
\begin{align*}
|b_{i,j}^{n}|&\leq |b_{i,j}^{j}|+\frac{16}{9}\sum_{\ell=j}^{n-1} \frac{q^{2\ell-i-j}(10)_{\ell+1}^2}{\eta_{i\ell}\eta_{j\ell}(20)_{\ell+1}} \\
&\leq \frac{4 q^{j-i}}{\eta_{ij}}\left(1+\frac{4}{9}\sum_{\ell=j}^{n-1}\frac{q^{2(\ell-j)}\eta_{ij}(10)_{\ell+1}^2}{\eta_{i\ell}\eta_{j\ell}(20)_{\ell+1}}\right).
\end{align*}
Since $(10)_{\ell+1}\leq (20)_{\ell+1}$, $(10)_{\ell+1}\leq (20)_\ell\leq \eta_{j\ell}$ and $\eta_{ij}\leq \eta_{i\ell}$ for $j\leq \ell\leq n-1$, we estimate this from above by
\[
\frac{4q^{j-i}}{\eta_{ij}}\left(1+\frac{4}{9}\sum_{\ell=j}^\infty q^{2(\ell-j)}\right)
=\frac{36}{5}\frac{q^{j-i}}{\eta_{ij}},
\]
proving the theorem.
\end{proof}
This proves Theorem \ref{thm:main} for piecewise linear splines, i.e. $k=2$.

\section{Piecewise quadratic splines}\label{sec:quad}
In this section, we apply Corollary \ref{prop:inv} to the case of piecewise quadratic splines to get geometric estimates for the entries of their Gram matrix inverse. 

\subsection{The Gram matrix} We now calculate the Gram matrix of B-splines of order three.
There is a standard formula for the inner product of B-splines of arbitrary order $k$ involving divided differences (see for instance \cite[Theorem 4.25]{Schumaker1981}), but in this section we prefer to calculate the Gram matrix for $k=3$ directly. The reason is a simplification of the general formula in our special case $k=3$.
As an easy consequence of the recursion formula for B-splines, we get the following
\begin{cor}\label{cor:spline3} The B-spline functions $N_i\equiv N_{i,3}$, $1\leq i\leq m$ of order $3$ corresponding to the knot sequence $\overline{\Delta}=(t_i)_{i=1}^{|\mu|}$ are given by
\begin{equation}\label{eq:spline3}
N_{i}(x)=\begin{cases}
\dfrac{(x-t_i)^2}{(20)_i(10)_i} & \text{if }x\in [t_i,t_{i+1}), \\
\dfrac{(x-t_i)(t_{i+2}-x)}{(20)_i(21)_i}+ \dfrac{(x-t_{i+1})(t_{i+3}-x)}{(31)_i(21)_i} & \text{if }x\in [t_{i+1},t_{i+2}), \\
\dfrac{(t_{i+3}-x)^2}{(31)_i(32)_i} & \text{if }x\in [t_{i+2},t_{i+3}), \\
0 & \text{otherwise}.
\end{cases}
\end{equation}
\end{cor}

\begin{lem}
Let $1\leq i\leq m-1$. Then we have 
\begin{align}
\int_{t_{i+1}}^{t_{i+2}}N_i(x)N_{i+1}(x)\dif x
&= \frac{(21)_i^2}{(31)_i}\left[\frac{1}{10}+\frac{(10)_i}{30(20)_i}+\frac{(32)_i}{5(31)_i}\right] \label{eq:scalar1}
\end{align}
and due to symmetry
\begin{align}
\int_{t_{i+2}}^{t_{i+3}}N_i(x)N_{i+1}(x)\dif x
&= \frac{(32)_i^2}{(31)_i}\left[\frac{1}{10}+\frac{(43)_i}{30(42)_i}+\frac{(21)_i}{5(31)_i}\right] \label{eq:scalar3}
\end{align}
\end{lem}

\begin{proof}
A straighforward calculation using Corollary \ref{cor:spline3} yields
\begin{equation*}
\int_{t_{i+1}}^{t_{i+2}}N_i(x)N_{i+1}(x)\dif x=\frac{(21)_i^2}{(31)_i}\left[\frac{(21)_i}{20(20)_i}+\frac{(10)_i}{12(20)_i}+\frac{1}{4}-\frac{(21)_i}{5(31)_i}\right].
\end{equation*}
An easy reformulation of this identity gives us \eqref{eq:scalar1}. Equation \eqref{eq:scalar3} 
 now follows by symmetry.
\end{proof}

\begin{prop}\label{prop:gramquadratic}
The entries of the Gram matrix $A=(a_{i,j})_{i,j=1}^m=(\langle N_i,N_j\rangle)_{i,j=1}^m$  of B-splines of order $3$ admit the following representation.
\begin{align}
\label{eq:diag2}a_{i,i+2}=a_{i+2,i}&=\frac{(32)_i^3}{30(31)_i(42)_i} & \text{if }1\leq i\leq m-2,\\
\label{eq:diag1}a_{i,i+1}=a_{i+1,i}&=\frac{(31)_i}{10}+\frac{(21)_i^2(10)_i}{30(20)_i(31)_i}+\frac{(32)_i^2(43)_i}{30(31)_i(42)_i} &\text{if } 1\leq i\leq m-1, \\
\label{eq:diag}a_{i,i}&= \frac{(30)_i}{5}-\frac{(30)_i(21)_i^2}{15(20)_i(31)_i} &\text{if } 1\leq i\leq m.
\end{align}
Furthermore, $a_{i,j}=0$ if $|i-j|>2$.
\end{prop}

\begin{proof}
Equation \eqref{eq:diag2} is a simple consequence of Corollary \ref{cor:spline3}. For equation \eqref{eq:diag1}, we observe that $\langle N_i,N_{i+1}\rangle $ is the sum of \eqref{eq:scalar1} and \eqref{eq:scalar3}. Thus, \eqref{eq:diag1} is a consequence of the identity
\begin{align*}
\frac{(21)^2}{(31)}\left[\frac{1}{10}+\frac{(32)}{5(31)}\right]+\frac{(32)^2}{(31)}\left[\frac{1}{10}+\frac{(21)}{5(31)}\right] &= \frac{(21)^2+(32)^2}{10(31)}+\frac{(21)(32)}{5(31)}= \frac{(31)}{10},
\end{align*}
where every bracket $(mn)$ is taken with respect to the subindex $i$.
Since $N_i$ has compact supoort and the B-splines form a partition of unity, we obtain
\begin{equation}\label{eq:gramhilf}
\Big\langle N_i,\sum_{l=i-2}^{i+2} N_l\Big\rangle=\langle N_i,1\rangle =\frac{(30)_i}{3}.
\end{equation}
The only unknown part in this equation is $\langle N_i,N_i\rangle $, so we use \eqref{eq:gramhilf} and simple calculations to complete the proof of the proposition.
\end{proof}
Observe that the Gram matrix $A$ is symmetric, 5-banded and totally positive (cf. Theorem \ref{thm:gramtp}), so in particular, we can apply Lemmas \ref{lem:bnnnquad} and \ref{lem:estbjnn} to $A$.

In the following, we use a computer algebra system, in our case \textsc{Mathematica} 8.0, to show that certain given polynomials have only nonnegative coefficients. This allows us to deduce that the given polynomial itself is nonnegative for positive arguments. In the following results, we need the matrix $A$ to be defined in \mathematica. This is done in Figure \ref{fig:defgram}.
\begin{figure}
\centering
\includegraphics[scale=0.95]{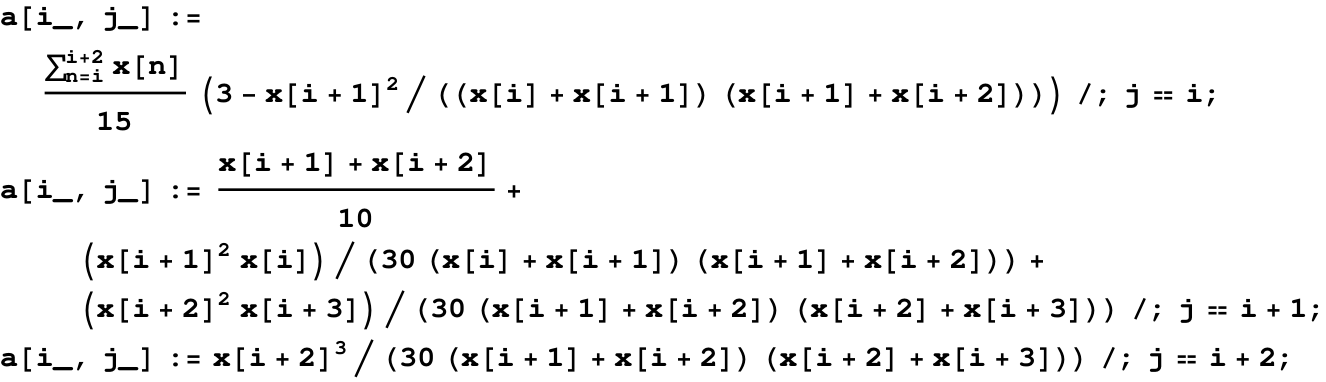}
\caption{Definition of the Gram matrix $A$ in \textsc{Mathematica}}
\label{fig:defgram}
\end{figure}
Furthermore we need a \mathematica\;expression for $\varphi_n$ to be defined in Proposition \ref{lem:bnnnfinal}. For this, we refer to Figure \ref{fig:phi}.

\begin{figure}
\centering
\includegraphics[scale=1]{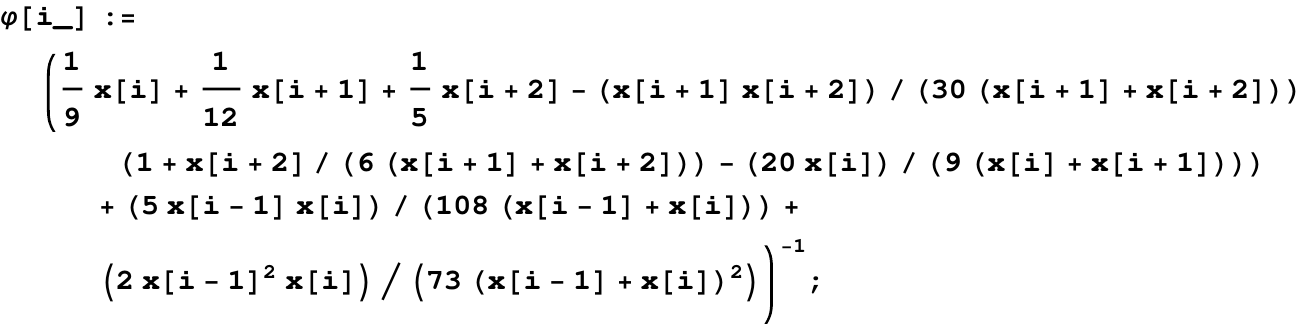}
\caption{Definition of the function $\varphi$ in \textsc{Mathematica}}
\label{fig:phi}
\end{figure}

The first thing to show is, as in Section \ref{sec:linear}, a suitable upper bound for $b_{nn}^n$.  This is the content of the following Proposition \ref{lem:bnnnfinal}.
\begin{prop}\label{lem:bnnnfinal}
Let $A$ be as in Proposition \ref{prop:gramquadratic} and $1\leq n\leq m$. Then the  element $b_{n,n}^n$ of $B_n=A_n^{-1}$ satisfies the estimate
\begin{equation}
\begin{aligned}
b_{n,n}^n&\leq \bigg(\frac{(10)}{9}+\frac{(21)}{12}+\frac{(32)}{5}-\frac{(21)(32)}{30(31)}\Big(1+\frac{(32)}{6(31)}-\frac{20(10)}{9(20)}\Big) \\
&+\frac{5(0,-1)(10)}{108(1,-1)}+\frac{2(0,-1)^2(10)}{73(1,-1)^2}\bigg)^{-1} =:\varphi_n,
\end{aligned}
\label{eq:bnnnquadgeneral}
\end{equation}
where every bracket $(ij)$ in this formula is taken with respect to the subindex $n$.
\end{prop}

\begin{proof}
We use Lemma \ref{lem:bnnnquad} and induction to prove the claimed estimate. Recall that Lemma \ref{lem:bnnnquad} stated that
\begin{equation}
\begin{aligned}
b_{n+1,n+1}^{n+1}\leq\;& \Big(a_{n+1,n+1}-b_{n,n}^n a_{n,n+1}\big(a_{n,n+1}-\frac{2a_{n,n}a_{n-1,n+1}}{a_{n-1,n}}\big)\\-&2\frac{a_{n,n+1}a_{n-1,n+1}}{a_{n-1,n}}
-a_{n-1,n+1}^2b_{n-1,n-1}^{n-1}(1+b_{n,n}^n b_{n-1,n-1}^{n-1} a_{n-1,n}^2)\Big)^{-1}
\end{aligned}
\label{eq:quadrgeneral1copy}
\end{equation}
for $2\leq n\leq m-1$. Thus to apply induction, we need to verify \eqref{eq:bnnnquadgeneral} for $n=1$ and $n=2$. First, we suppose that $n=1$; here we have $b_{1,1}^1=a_{1,1}^{-1}=5/(30)_1$ by \eqref{eq:diag} and the fact that $(20)_1=0$. For $n=2$, the formula $b_{2,2}^2=(a_{2,2}-b_{1,1}^1 a_{1,2}^2)^{-1}$ holds in view of Corollary \ref{prop:inv}.  We thus obtain by \eqref{eq:diag1}, \eqref{eq:diag}, $(20)_1=0$ and some elementary calculations
\begin{equation}
b_{2,2}^2=\bigg(\frac{(21)_2}{12}+ \frac{(32)_2}{5}- \frac{(21)_2(32)_2}{30(31)_2}\Big(1+ \frac{(32)_2}{(6(31)}\Big) \bigg)^{-1}.
\label{eq:b222}
\end{equation}
The expression for $b_{1,1}^1$ and \eqref{eq:b222} are special cases of \eqref{eq:bnnnquadgeneral} since $(20)_1=0$. Thus, the lemma is proved for $n=1$ and $n=2$. 

Before we proceed with the actual proof of \eqref{eq:bnnnquadgeneral}, we show that the term $a_{n,n+1}-\frac{2a_{n,n}a_{n-1,n+1}}{a_{n-1,n}}$, appearing in \eqref{eq:quadrgeneral1copy}, is nonnegative. It is equivalent to show that the expression $a_{n,n+1}a_{n-1,n}-2a_{n,n}a_{n-1,n+1}$ is nonnegative. 
This is done using the \mathematica-code of Figure \ref{fig:estgram}. 
The method of proof is as follows. 
\begin{enumerate}

\item Define the rational function $p=a_{n,n+1}a_{n-1,n}-2a_{n,n}a_{n-1,n+1}$ depending on the variables $x[1]=(10)_{n-1},\,x[2]=(21)_{n-1},\dots,\,x[5]=(54)_{n-1}$.
\item Determine the denominator $d$ of the rational function $p$ as $900(x[1]+x[2])(x[2]+x[3])^2 (x[3]+x[4])^2(x[4]+x[5])$ and observe that $d$ is positive for positive arguments.
\item Calculate the coefficients of the polynomial $q:=d\cdot p$ and verify that no coefficient of $q$ is negative.
\end{enumerate}

\begin{figure}
\centering
\includegraphics[scale=1]{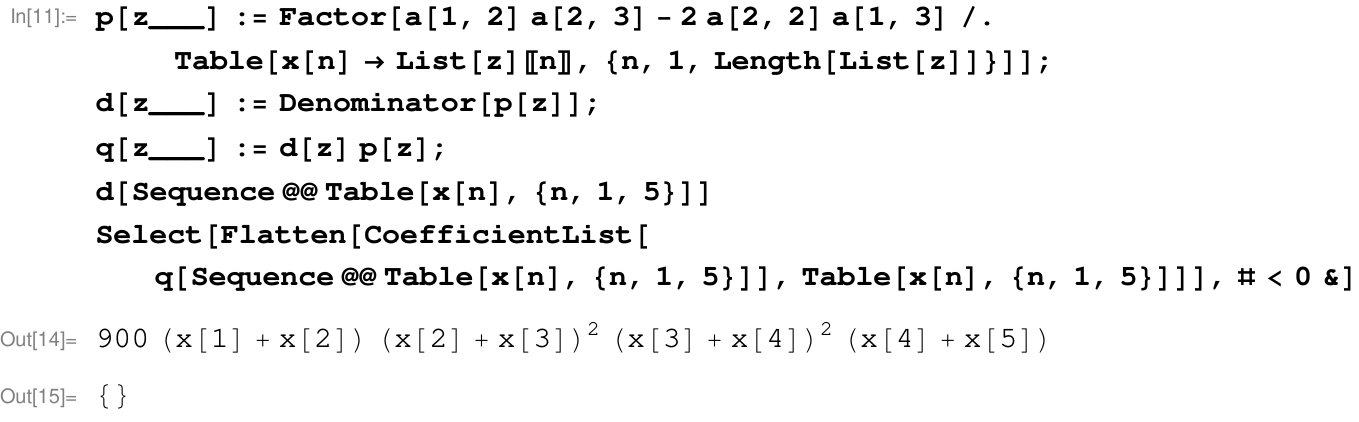}
\caption{Proof of the inequality $a_{n,n+1}-2a_{nn}a_{n-1,n+1}/a_{n-1,n}\geq 0$}
\label{fig:estgram}
\end{figure}

Now we continue with the proof of \eqref{eq:bnnnquadgeneral}. Since every entry $a_{ij}$ of the Gram matrix A is nonnegative and the matrices $B_{n-1},B_n$ are checkerboard, the argument of the last paragraph shows that a sufficient condition for \eqref{eq:bnnnquadgeneral} to be true for all $1\leq n\leq m$ is the following recursive inequality for $\varphi_n$, $2\leq n\leq m-1$.
\begin{equation}
\begin{aligned}
\Big(&a_{n+1,n+1}-\varphi_n a_{n,n+1}\big(a_{n,n+1}-\frac{2a_{n,n}a_{n-1,n+1}}{a_{n-1,n}}\big)\\&-2\frac{a_{n,n+1}a_{n-1,n+1}}{a_{n-1,n}}
-a_{n-1,n+1}^2\varphi_{n-1}(1+\varphi_n \varphi_{n-1} a_{n-1,n}^2)\Big)^{-1}\leq \varphi_{n+1}.
\end{aligned}
\label{eq:recphi}
\end{equation}
The proof that this inequality is true for all choices of  $(0,-1)_{n-1},\,(10)_{n-1},\dots,\,(54)_{n-1}$ is equivalent to prove that the rational function $p$ defined as
\begin{equation}
\label{eq:phin}
\begin{aligned}
\varphi_{n}^{-1}\varphi_{n-1}^{-1} a_{n-1,n}&\bigg(\Big(a_{n+1,n+1}-\varphi_n a_{n,n+1}\big(a_{n,n+1}-\frac{2a_{n,n}a_{n-1,n+1}}{a_{n-1,n}}\big)\\&-2\frac{a_{n,n+1}a_{n-1,n+1}}{a_{n-1,n}}
-a_{n-1,n+1}^2\varphi_{n-1}(1+\varphi_n \varphi_{n-1} a_{n-1,n}^2)\Big)-\varphi_{n+1}^{-1}\bigg)
\end{aligned}
\end{equation}
and depending on the six variables $(0,-1)_{n-1},\,(10)_{n-1},\dots,\,(54)_{n-1}$ is nonnegative for nonnegative arguments. This is done in the \mathematica-code of Figure \ref{fig:bnnn} using the following steps.
\begin{enumerate}
\item Define the rational function $p$ to be the expression in \eqref{eq:phin} depending on the variables $x[1]=(10)_{n-2},\,x[2]=(21)_{n-2},\dots,\,x[5]=(54)_{n-2},\,x[6]=(65)_{n-2}$.
\item Determine the denominator $d$ of the rational function $p$ as an integer multiple of $(x[1]+x[2])^4(x[2]+x[3])^5 (x[3]+x[4])^8(x[4]+x[5])^5 (x[5]+x[6])^2$ and observe that $d$ is positive for positive arguments.
\item Calculate the coefficients of the polynomial $q:=d\cdot p$ and verify that no coefficient of $q$ is negative.
\end{enumerate}
\begin{figure}
\centering
\caption{Proof of the estimate for $b_{nn}^n$}
\includegraphics[scale=1]{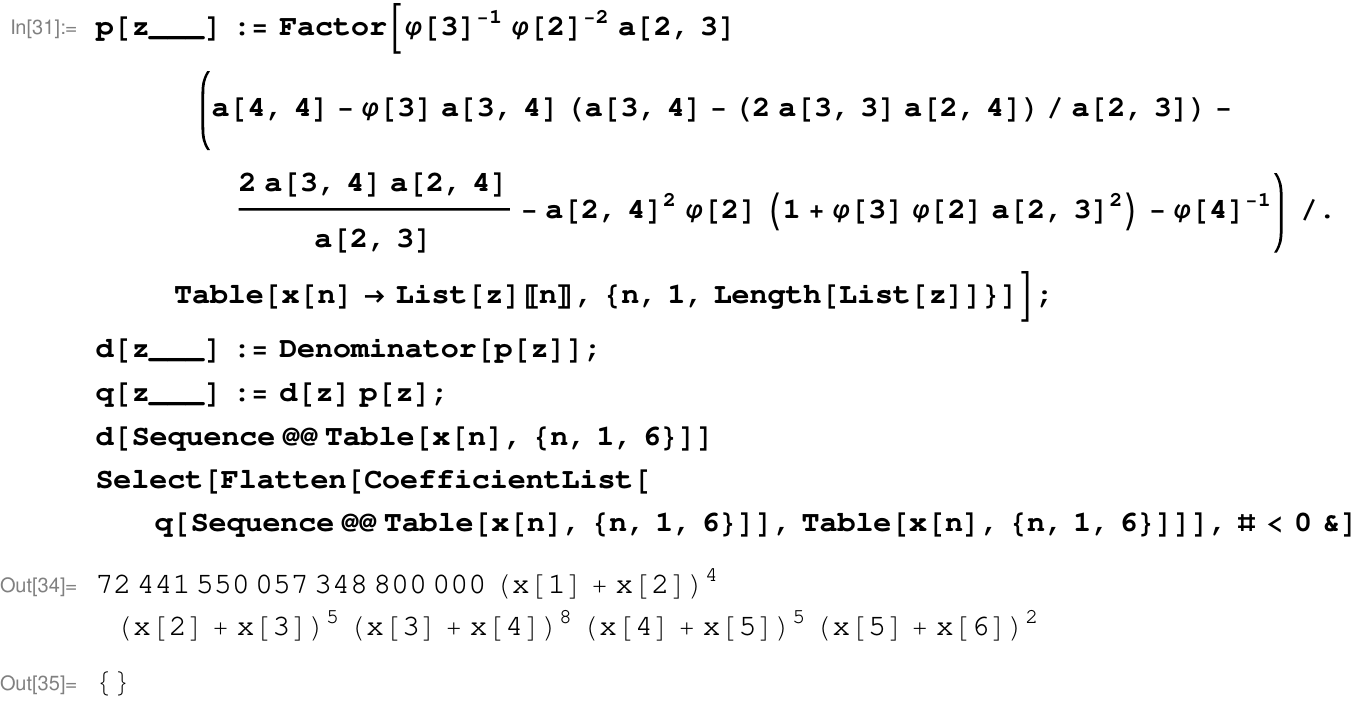}
\label{fig:bnnn}
\end{figure}
This proves the assertion of the proposition
\end{proof}

\begin{rem}\label{rem:worsebnnn}
It is an easy consequence of the above lemma and the inequality 
\[
\frac{(21)(32)}{(31)}\Big(1+\frac{(32)}{6(31)}\Big)\leq (32),
\]
that we also have the estimate
\begin{equation}\label{eq:bnnnworse}
b_{n,n}^n\leq \bigg(\frac{(10)}{9}+\frac{(21)}{12}+\frac{(32)}{6}\bigg)^{-1}=:\psi_n
\end{equation}
for all $1\leq n\leq m$. The \mathematica\ expression of $\psi_n$ is given by Figure \ref{fig:psi}.
\end{rem}
\begin{figure}
\centering
\includegraphics[scale=1]{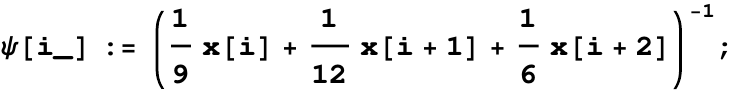}
\caption{Definition of the function $\psi$}
\label{fig:psi}
\end{figure}

\begin{lem}\label{lem:badgeom}
Let $A$ be as in Proposition \ref{prop:gramquadratic} and $2\leq n\leq m$. Then we have the estimate
\begin{equation*}
b_{n,n}^n a_{n-1,n} \leq \frac{6}{5} \frac{(20)_n}{(30)_n}.
\end{equation*}
\end{lem}

\begin{proof}
By Remark \ref{rem:worsebnnn}, it suffices to show that 
\[
\psi_n a_{n-1,n}\leq \frac{6}{5}\frac{(20)_n}{(30)_n}.
\]
We apply the same method of proof as in the above proposition and proceed with the \mathematica-code of Figure \ref{fig:1komma2} using the following steps.
\begin{enumerate}
\item Define the rational function $p=\frac{6}{5}\frac{(20)_n}{(30)_n}-\psi_n a_{n-1,n}$ depending on the variables $x[1]=(10)_{n-1},x[2]=(21)_{n-1},x[3]=(32)_{n-1},x[4]=(43)_{n-1}$.
\item Determine the denominator $d$ of the rational function $p$ as $5(x[1]+x[2])$ $(x[2]+x[3])(x[3]+x[4])(x[2]+x[3]+x[4])(4x[2]+3x[3]+6x[4])$ and observe that $d$ is positive for positive arguments.
\item Calculate the coefficients of the polynomial $q:=d\cdot p$ and verify that no coefficient of $q$ is negative.\qedhere
\end{enumerate}
\begin{figure}
\centering
\includegraphics[scale=1]{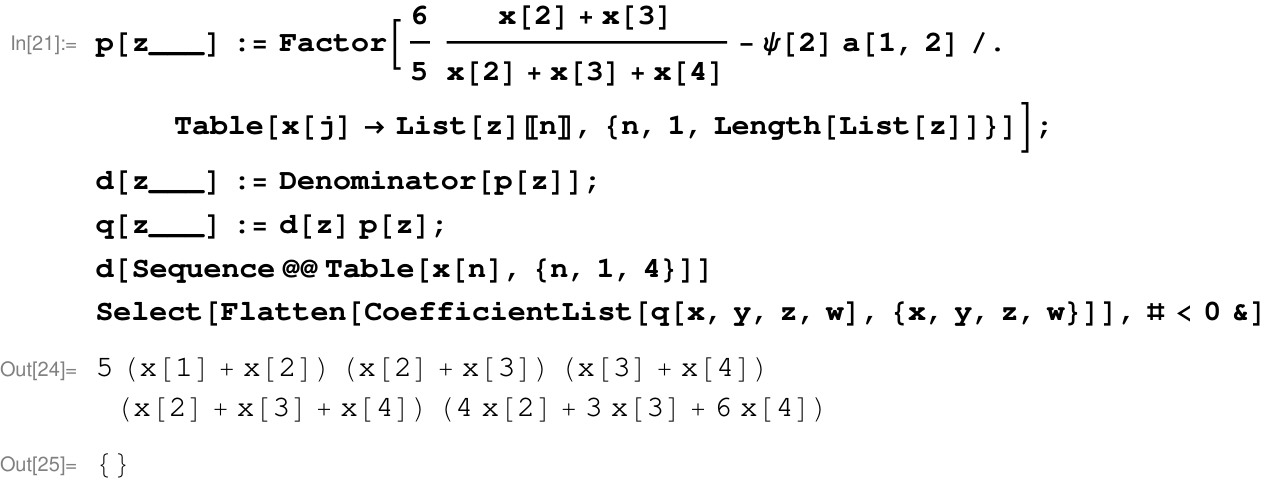}
\caption{Proof of the estimate for $b_{n,n}^na_{n-1,n}$}
\label{fig:1komma2}
\end{figure}
\end{proof}

Let $\theta_n:= b_{n,n}^n\big(a_{n-1,n}-\frac{a_{n-2,n}a_{n-1,n-1}}{a_{n-2,n-1}}\big)$. This expression is important, since it is used in Lemma \ref{lem:estbjnn} to estimate $|b_{j,n}^n|$ for $1\leq j\leq n-1$. In the following lemma, we estimate the product of two consecutive values of $\theta_n$. We will use this result in the proof of Proposition \ref{thm:lastlinequadrfinal} to obtain explicit estimates for $|b_{j,n}^n|$.
\begin{lem}\label{lem:consec}
Let $A$ be as in Proposition \ref{prop:gramquadratic} and $3\leq n\leq m-1$. Then we have that
\begin{equation}
\theta_n\theta_{n+1}\leq \frac{87}{100}\frac{(20)_n}{(30)_n}\frac{(20)_{n+1}}{(30)_{n+1}}
\label{eq:consec}.
\end{equation}
\end{lem}
\begin{proof}
We show that the rational function $p$, defined as
\begin{equation}\label{eq:theta}
\frac{87}{100}-\frac{(30)_n}{(20)_n}\frac{(30)_{n+1}}{(20)_{n+1}}\theta_n\theta_{n+1},
\end{equation}
depending on the variables $x[1]=(10)_{n-2},x[2]=(21)_{n-2},\dots,x[6]=(65)_{n-2}$ is nonnegative for nonnegative arguments. This is done with the \mathematica-code in Figure \ref{fig:theta} using the steps
\begin{enumerate}
\item Define the rational function $p$ as in \eqref{eq:theta}.
\item Determine the denominator $d$ of the rational function $p$ and verify that no coefficient of the polynomial $d$ is negative.
\item Determine the coefficients of the polynomial $q:=d\cdot p$ and verify that no coefficient of $q$ is negative.\qedhere
\end{enumerate}
\begin{figure}
\centering
\includegraphics[scale=1]{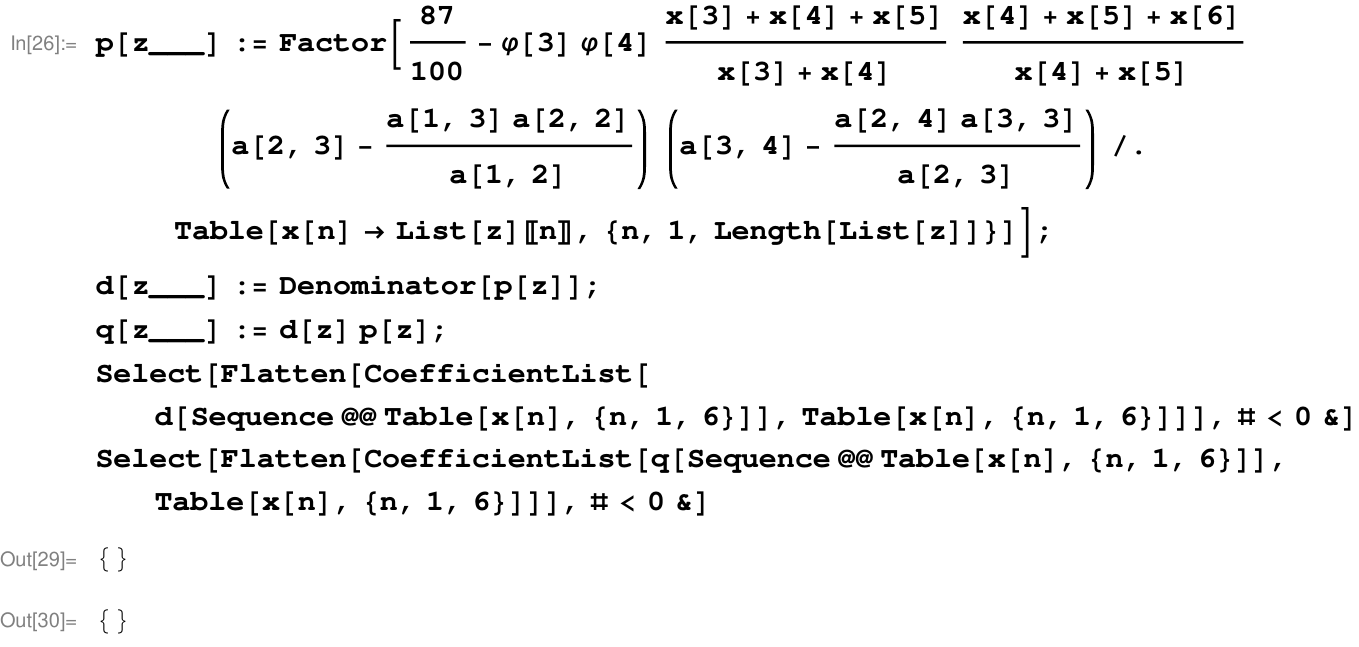}
\caption{Proof of the estimate for the product of two consecutive values of $\theta_n$}
\label{fig:theta}
\end{figure}
\end{proof}
\begin{prop}\label{thm:lastlinequadrfinal}
Let $A$ be as in Proposition \ref{prop:gramquadratic}, $1\leq n\leq m$ and $1\leq j\leq n$. Then we have the estimate
\[
|b_{j,n}^n|\leq C\frac{q^{n-j}}{\eta_{jn}}\quad\text{with }q=(87/100)^{1/2},\ C=12 q^{-2} \Big(\frac{6}{5}\Big)^2.
\]
\end{prop}
\begin{proof}
We see in the first place that Remark \ref{rem:worsebnnn} yields the estimate $b^{n}_{n,n}\leq 12/(30)_n$ and thus the assertion of the theorem in the case $j=n$. If $j\leq n-1$, we get from Lemma \ref{lem:estbjnn}
\begin{equation}\label{eq:lastlinequadrfinal:1}
|b_{j,n}^n|\leq  b_{j,j}^j b_{j+1,j+1}^{j+1} a_{j,j+1}\prod_{\ell=j+2}^n \theta_\ell.
\end{equation}
Now assume that $n-(j+2)+1=n-j-1$ is odd. By Lemmas \ref{lem:badgeom} and \ref{lem:consec} we conclude from \eqref{eq:lastlinequadrfinal:1} that
\begin{equation}
|b_{j,n}^n| \leq \frac{12}{(30)_j}\Big(\frac{6}{5}\Big)^2 q^{n-j-2}\prod_{\ell=j+1}^n \frac{(20)_\ell}{(30)_\ell}.
\end{equation}
We use induction and inequality \eqref{eq:etaineq} for $\eta$ to obtain the final estimate
\begin{equation}\label{eq:lastlinequadrfinal:2}
|b_{j,n}^n|\leq 12 q^{-2} \Big(\frac{6}{5}\Big)^2 \frac{q^{n-j}}{\eta_{jn}}.
\end{equation}
If we assume that $n-(j+2)+1$ is even, the same reasoning yields the estimate 
\begin{equation}\label{eq:lastlinequadrfinal:3}
|b_{j,n}^n|\leq 12 q^{-1} \frac{6}{5} \frac{q^{n-j}}{\eta_{jn}}.
\end{equation}
Inequalities \eqref{eq:lastlinequadrfinal:2} and \eqref{eq:lastlinequadrfinal:3} together now prove the proposition.
\end{proof}

The passage to estimates of expressions $|b_{i,j}^n|$ for $1\leq i,j\leq n$ is now an analogous calculation as in Theorem \ref{thm:lingeom} and carried out in the following

\begin{thm}\label{thm:lastquadratic}
Let $A$ be as in Proposition \ref{prop:gramquadratic} and $1\leq n\leq m$. Then the entries $b_{i,j}^n$ of the matrix $B_n$ satisfy the estimate
\begin{equation}
|b_{i,j}^n|\leq C_1 \frac{q^{|i-j|}}{\eta_{ij}}\quad\text{for all }1\leq i,j\leq n,
\label{eq:lastquadratic}
\end{equation}
where $C_1=C\big(1+ \frac{12}{75}\frac{C}{1-q^2}\big)$ and $C,q$ are as in Proposition \ref{thm:lastlinequadrfinal}.
\end{thm}

\begin{proof}
Since $B_n$ is symmetric and the case $j=n$ was treated in Proposition \ref{thm:lastlinequadrfinal}, it suffices to consider the cases $i\leq j\leq n-1$. Equations \eqref{eq:itquadr}\,--\,\eqref{eq:Cnquadr} yield the formula
\[
b_{i,j}^n=b_{i,j}^{n-1}+b_{i,n}^n b_{j,n}^n/b_{n,n}^n.
\]
By Lemma \ref{lem:estbjnn}, inequality \eqref{eq:bjnnbad}, we obtain
\[
|b_{i,j}^n|\leq |b_{i,j}^{n-1}|+b_{n,n}^n a_{n-1,n}^2 |b_{i,n-1}^{n-1} b_{j,n-1}^{n-1}|.
\]
Applying Lemma \ref{lem:badgeom} and Proposition \ref{thm:lastlinequadrfinal} we estimate further
\[
|b_{i,j}^n|\leq |b_{i,j}^{n-1}|+C^2\frac{6}{5}a_{n-1,n}\frac{(20)_n}{(30)_n}\frac{q^{2n-2-i-j}}{\eta_{i,n-1}\eta_{j,n-1}}.
\]
By induction, we get
\begin{align*}
|b_{i,j}^n|&\leq |b_{i,j}^j| +\frac{6}{5}C^2\sum_{\ell=j}^{n-1} a_{\ell,\ell+1} \frac{(20)_{\ell+1}}{(30)_{\ell+1}}\frac{q^{2\ell-i-j}}{\eta_{i\ell}\eta_{j\ell}} \\
&\leq C \frac{q^{j-i}}{\eta_{ij}}\Big(1+ \frac{6}{5}C\sum_{\ell=j}^{n-1}\frac{(20)_{\ell+1}}{(30)_{\ell+1}}a_{\ell,\ell+1} \frac{q^{2(\ell-j)}\eta_{ij}}{\eta_{i\ell}\eta_{j\ell}}\Big).
\end{align*}
An easy consequence of formula \eqref{eq:diag1} for $a_{\ell,\ell+1}$ is that $a_{\ell,\ell+1}\leq 2(31)_\ell/15\leq 2\eta_{i\ell}/15$. This and the obvious inequalities $\eta_{ij}\leq \eta_{i\ell}, (20)_{\ell+1}\leq (30)_{\ell+1}$ for $\ell$ as in the above sum give the final estimate
\[
|b_{i,j}^n| \leq C \frac{q^{j-i}}{\eta_{ij}}\Big(1+ \frac{12}{75}C\sum_{\ell=j}^{\infty}q^{2(\ell-j)}\Big)= C \frac{q^{j-i}}{\eta_{ij}}\Big(1+ \frac{12}{75}\frac{C}{1-q^2}\Big). \qedhere
\]
\end{proof}
This proves Theorem \ref{thm:main} for piecewise quadratic splines, i.e. $k=3$.

\subsubsection*{\bfseries \emph{Acknowledgments}}
It is a pleasure to thank Zbigniew Ciesielski, Anna Kamont and Paul Müller for their comments on a draft of this paper. 

The author is supported by the Austrian Science Fund FWF project P 23987-N18. A part of this research was performed while the author was visiting the Institute of Mathematics of the Polish Academy of Sciences in Sopot. He thanks the Institute for its hospitality and excellent working conditions. These stays were supported by MNiSW grant N N201 607840.

\bibliographystyle{amsalpha}
\bibliography{invgramestimate}

\end{document}